\documentclass[11pt]{article}
\usepackage[margin=1.15in]{geometry}
\usepackage[T1]{fontenc}
\usepackage{lmodern}
\usepackage{amsmath,amssymb,amsthm}
\usepackage{parskip}
\usepackage{microtype}
\usepackage{hyperref}
\hypersetup{
  unicode=true,
  colorlinks=true,
  linkcolor=blue,
  citecolor=blue,
  urlcolor=blue,
  pdftitle={A Resolution of the de Bruijn--Erd\H{o}s Consecutive-Gap Problem},
  pdfauthor={Samuel Korsky}
}
\newtheorem{theorem}{Theorem}[section]
\newtheorem{lemma}[theorem]{Lemma}
\newtheorem{proposition}[theorem]{Proposition}
\theoremstyle{remark}
\newtheorem{remark}[theorem]{Remark}
\numberwithin{equation}{section}
\newcommand{\T}{\mathbb{T}}
\newcommand{\R}{\mathbb{R}}
\newcommand{\N}{\mathbb{N}}
\title{A Resolution of the de Bruijn--Erd\H{o}s Consecutive-Gap Problem}
\author{Samuel Korsky}
\date{September 8, 2026}
\begin{document}
\maketitle
\begin{abstract}
\noindent
Let $(x_n)_{n\geq1}$ be a sequence of distinct points on the unit circle.
An $r$-span is the total length of $r$ consecutive gaps determined by the
inserted points. Write $M_n^{(r)}$ and $m_n^{(r)}$ for the largest and smallest
$r$-spans after the first $n$ insertions. We prove that there is an absolute
constant $c>0$ such that, for every sufficiently large $r$,
\[
  \limsup_{n\to\infty}\bigl(nM_n^{(r)}-r\bigr)
  \geq c\sqrt{\log r},
  \qquad
  \limsup_{n\to\infty}\bigl(r-nm_n^{(r)}\bigr)
  \geq c\sqrt{\log r},
\]
and
\[
  \limsup_{n\to\infty}\frac{M_n^{(r)}}{m_n^{(r)}}
  \geq 1+\frac{\log r}{100r}.
\]
Thus all three asymptotic conjectures made by de Bruijn and Erd\H{o}s in
1949 are resolved. The ratio bound matches the upper bound of Cl\'ement and
Steinerberger up to an absolute constant and answers a question of
Brethouwer. The proofs compare interval counts at nearby times. Pointwise
control leads to a one-dimensional sequence-discrepancy argument for the
ratio, while averaged control and Hal\'asz's planar $L^1$ discrepancy theorem
give the two one-sided conclusions.
\end{abstract}
\section{Introduction}
Let $(x_n)_{n\geq1}$ be a sequence of distinct points on the circle
$\T=\R/\mathbb{Z}$. After the first $n$ points are inserted, write the gap
lengths in cyclic order as $g_1^{(n)},\ldots,g_n^{(n)}$. An \emph{$r$-span}
is the total length of $r$ consecutive gaps. For fixed $r\geq1$ and $n\geq r$,
set
\[
  M_n^{(r)}=\max_i\sum_{j=0}^{r-1}g_{i+j}^{(n)},
  \qquad
  m_n^{(r)}=\min_i\sum_{j=0}^{r-1}g_{i+j}^{(n)},
\]
with indices taken cyclically. The average $r$-span is $r/n$, since each
gap occurs in exactly $r$ of the $n$ spans. Throughout, $\log$ denotes the
natural logarithm.

For a sequence $X=(x_n)$, put
\[
  \overline A_r(X)=\limsup_{n\to\infty}nM_n^{(r)},
  \qquad
  \underline A_r(X)=\liminf_{n\to\infty}nm_n^{(r)},
\]
and
\[
  \mu_r(X)=\limsup_{n\to\infty}
  \frac{M_n^{(r)}}{m_n^{(r)}}.
\]
The optimal universal quantities are
\[
  \overline A_r=\inf_X\overline A_r(X),
  \qquad
  \underline A_r=\sup_X\underline A_r(X),
  \qquad
  \mu_r=\inf_X\mu_r(X),
\]
where $X$ ranges over all sequences of distinct points on $\T$.
The average-span identity gives $\overline A_r\geq r$ and
$\underline A_r\leq r$.

De Bruijn and Erd\H{o}s~\cite{debruijn-erdos} obtained universal bounds for
all three quantities. In the final mathematical paragraph of their paper
\cite[p.~17]{debruijn-erdos}, they conjectured that
\[
  \overline A_r-r\longrightarrow\infty,
  \qquad
  r-\underline A_r\longrightarrow\infty,
  \qquad
  r(\mu_r-1)\longrightarrow\infty
  \qquad(r\to\infty).
\]
This three-part conjecture is also recorded as Erd\H{o}s
Problem~\#1221~\cite{erdos-problem-1221}.

For the third assertion, de Bruijn and Erd\H{o}s proved
\[
  \mu_r\geq1+\frac1r.
\]
Thus their conjecture asks whether $1/r$ can be replaced by $f(r)/r$ for
some function $f(r)$ tending to infinity. This lower bound is sharp when
$r=1$.

Cl\'ement and Steinerberger~\cite{clement-steinerberger} proved that, for all
sufficiently large $r$,
\[
  \mu_r\leq1+\frac{C\log r}{r}
\]
with an absolute constant $C$. Their examples are the binary van der Corput
sequence and the sequence of multiples of the golden ratio, taken modulo
one. Brethouwer~\cite[Section~3.3.1, Question~3]{brethouwer} had asked whether
a lower bound of the same order holds.

Recent work has improved the ratio bounds for fixed $r$. The author proved the
lower bound $1+r/(r^2-1)$ for $r\geq2$ in~\cite{korsky}. The author's
unpublished manuscript~\cite{korsky-entropy} reports the entropy-potential
bound $1+\log_2(1+1/r)$; that bound is not used here. Bevan~\cite{bevan}
gives constructions and explicit upper bounds for small values of $r$.
Related questions with a prescribed finite number of insertions were
studied by DeLeo, Henderschedt, and Wells~\cite{deleo-henderschedt-wells}.

Our main result resolves all three asymptotic questions quantitatively.
\begin{theorem}\label{thm:main}
There are absolute constants $c>0$ and $r_0\in\N$ such that, for every
integer $r\geq r_0$
and every sequence of distinct points on $\T$,
\begin{align*}
  \limsup_{n\to\infty}\bigl(nM_n^{(r)}-r\bigr)
  &\geq c\sqrt{\log r},\\
  \limsup_{n\to\infty}\bigl(r-nm_n^{(r)}\bigr)
  &\geq c\sqrt{\log r},
\end{align*}
and
\[
  \limsup_{n\to\infty}\frac{M_n^{(r)}}{m_n^{(r)}}
  \geq1+\frac{\log r}{100r}.
\]
\end{theorem}
Consequently,
\[
  \overline A_r-r\geq c\sqrt{\log r},
  \qquad
  r-\underline A_r\geq c\sqrt{\log r},
  \qquad
  \mu_r-1\geq\frac{\log r}{100r}
\]
for all sufficiently large $r$. This proves the three conjectures above.
Together with the upper bound of Cl\'ement and Steinerberger, the last
inequality also gives
\[
  \frac{\log r}{100r}\leq\mu_r-1\leq\frac{C\log r}{r}
\]
for all sufficiently large $r$, determining the order of $\mu_r-1$ and
answering Brethouwer's question.

Some formulations use the average length of the $r$ gaps, namely
\[
  \widehat M_n^{(r)}=\frac{M_n^{(r)}}r,
  \qquad
  \widehat m_n^{(r)}=\frac{m_n^{(r)}}r.
\]
In that notation, the first two quantities in Theorem~\ref{thm:main} are
exactly
\[
  \limsup_{n\to\infty}r\bigl(n\widehat M_n^{(r)}-1\bigr),
  \qquad
  \limsup_{n\to\infty}r\bigl(1-n\widehat m_n^{(r)}\bigr).
\]
\paragraph{Proof outline.}
The three conclusions come from two versions of the same comparison idea.

For the ratio bound, suppose that all $r$-spans have nearly the same length
at every sufficiently late time. Moving a point forward by $kr$ places in
cyclic order then moves
it by approximately $kr/n$. We compare this move with a backward move at a
nearby time. Cyclic moves are bijections, and the point sets are nested, so
the resulting maps are injective. Averaging these comparisons gives bounds
for the number of points in shorter intervals.
The quantitative mechanism is as follows. If the normalized span error is
at most $A$, in the sense of~\eqref{eq:span-assumption}, we obtain counting
error at most
\[
  3A+O\!\left(\frac{A}{\log(r/A)}\right)
  \quad\text{on intervals of length at most }\frac Sn,
  \qquad
  S=\frac{\sqrt{Ar}}{\log^2(r/A)}.
\]
We choose one such interval containing $L=\lfloor S\rfloor$ points and list
these points in their order of insertion. The same bound controls every
prefix of the rescaled list. A quantitative
form of Schmidt's discrepancy theorem~\cite{schmidt}, due to
Larcher~\cite{larcher}, forces this error to be at least $(\log L)/16$.
In the relevant range $1\leq A\leq\log r$, we have
$\log L=\tfrac12\log r+O(\log\log r)$. Consequently,
\[
  (3+o(1))A\geq\left(\frac1{32}+o(1)\right)\log r.
\]
The final proof uses the strict inequality $1/96>1/100$.

For either one-sided assertion, the missing side is recovered after
averaging over the circle. At an integer time, the $r$-spans have mean
$r/n$. Thus a one-sided error at most $A/n$ implies
\[
  \sum_i\left|S_i(n)-\frac rn\right|\leq2A.
\]
The cyclic-walk comparison then gives spatial $L^1$, rather than pointwise,
control of short-interval counting error. We localize the points in a moving
short interval and record insertion time as a second coordinate. Hal\'asz's
planar $L^1$ discrepancy theorem~\cite{halasz} forces an error of order
$\sqrt{\log L}$. Taking $L$ of polynomial size in $r$ proves the first two
bounds in Theorem~\ref{thm:main}.
\section{Comparing Interval Counts}
Fix $r\in\N$. For real $t\geq1$, write
\[
  P_t=\{x_1,\ldots,x_{\lfloor t\rfloor}\},
  \qquad
  N_t(I)=\#(P_t\cap I).
\]
Real times let us choose the comparisons without repeated integer
rounding. The point sets are nested: $P_s\subseteq P_t$ whenever $s\leq t$.
We assume in this section and the next that, for every sufficiently large
$t$, there are $a_t,b_t\geq0$ such that every $r$-span $S_i(t)$ satisfies
\begin{equation}\label{eq:span-assumption}
  \frac{r-a_t}{t}\leq S_i(t)\leq\frac{r+b_t}{t},
  \qquad a_t+b_t\leq A,
\end{equation}
where $A\geq1$ is fixed. We derive this assumption from the ratio hypothesis
in Section~\ref{sec:ratio-proof}. Keeping the sum $a_t+b_t$ under control,
rather than
bounding the two errors separately by $A$, is what gives the coefficient
$3A$ below.
All intervals are oriented half-open circle intervals. We use lifts to
$\R$ when translating endpoints or measuring cyclic displacements. All
times are taken large enough that the intervals used below have length
less than one.
For $D>0$, define
\[
  U_t(D)=\frac1D\sup_{x\in\T}N_t((x,x+D/t]),
  \qquad
  V_t(D)=\frac1D\inf_{x\in\T}N_t((x,x+D/t]).
\]
Thus $U_t(D)$ and $V_t(D)$ compare the largest and smallest counts in
intervals of length $D/t$ with $D$. We write $(z)_+=\max\{z,0\}$.
\begin{lemma}\label{lem:comparison}
Fix $D,E>0$ and an integer $k\geq1$, and put $q=E/(kr)$. If $q<1$, then,
for all sufficiently large $t$,
\begin{align}
  U_t(D)
  &\leq
  \left(1+\frac{3kA}{D}\right)(1+q)\,
  U_{(1+q)t}(E),\label{eq:comparison-upper}\\
  V_t(D)
  &\geq
  \left(1-q-\frac{kA}{D}(3-q)\right)_+
  V_{(1-q)t}(E).\label{eq:comparison-lower}
\end{align}
For fixed $E$ and $k$, the time threshold can be chosen uniformly for $D$
in any bounded range.
\end{lemma}
\begin{proof}
For each time $s$, let $F_s$ and $B_s$ be the forward and backward cyclic
moves by $kr$ places in $P_s$. Both are bijections, and $B_s=F_s^{-1}$.
Summing $k$ consecutive $r$-spans shows that the displacement of $F_s$
lies in
\[
  \left[\frac{kr-ka_s}{s},\frac{kr+kb_s}{s}\right].
\]
The corresponding backward displacement is the negative of a forward
displacement at the same time. In particular, a backward move does not
reverse the insertion process.
\paragraph{The upper bound.}
Put $t_+=(1+q)t$ and $\ell=E/t_+$. For each $0\leq u\leq\ell$, choose
$s=s(u)$ so that
\[
  \frac{kr}{t}-\frac{kr}{s}=u.
\]
Then $t\leq s\leq t_+$. Consider the composition
\[
  P_t\xrightarrow{F_t}P_t
  \hookrightarrow P_s\xrightarrow{B_s}P_s
  \hookrightarrow P_{t_+}.
\]
Every arrow is injective. The displacement of the composition differs
from $u$ by a quantity in
\[
  \left[-\frac{ka_t}{t}-\frac{kb_s}{s},
         \frac{kb_t}{t}+\frac{ka_s}{s}\right]
  \subseteq\frac{k}{t}[-a_t-A,b_t+A].
\]
Let $I$ have length $D/t$. Extend $I$ to the left by $k(a_t+A)/t$ and to
the right by $k(b_t+A)/t$, obtaining an interval $J$ with
\[
  |J|=\frac{D+k(a_t+b_t+2A)}{t}\leq\frac{D+3kA}{t}.
\]
The injection sends $P_t\cap I$ into $P_{t_+}\cap(J+u)$. Hence
\[
  N_t(I)\leq N_{t_+}(J+u)
  \qquad(0\leq u\leq\ell).
\]
Averaging and interchanging the integrations gives
\[
  \begin{aligned}
    N_t(I)\ell
    &\leq\int_0^\ell N_{t_+}(J+u)\,du\\
    &=\int_J N_{t_+}((v,v+\ell])\,dv\\
    &\leq |J|\,E\,U_{t_+}(E).
  \end{aligned}
\]
The half-open endpoint conventions do not affect these integrals.
Since $\ell=E/t_+$, division by $D\ell$ proves
\eqref{eq:comparison-upper}.
\paragraph{The lower bound.}
Put $t_-=(1-q)t$ and $\ell=E/t_-$. For $0\leq u\leq\ell$, choose
$s=s(u)$ so that
\[
  \frac{kr}{s}-\frac{kr}{t}=u.
\]
Then $t_-\leq s\leq t$. This time use the injection
\[
  P_{t_-}\hookrightarrow P_s\xrightarrow{B_s}P_s
  \hookrightarrow P_t\xrightarrow{F_t}P_t.
\]
Its displacement differs from $-u$ by a quantity in
\[
  \left[-\frac{ka_t}{t}-\frac{kA}{t_-},
         \frac{kb_t}{t}+\frac{kA}{t_-}\right].
\]
Move the left endpoint of $I$ to the right by $ka_t/t+kA/t_-$ and its
right endpoint to the left by $kb_t/t+kA/t_-$. Call the resulting interval
$J$, interpreted as empty if its length is nonpositive. Then
\[
  |J|\geq
  \left(\frac Dt-\frac{kA}{t}-\frac{2kA}{t_-}\right)_+,
  \qquad
  N_{t_-}(J+u)\leq N_t(I).
\]
Averaging as before yields
\[
  N_t(I)\ell\geq |J|\,E\,V_{t_-}(E).
\]
After division by $D\ell$, the coefficient is at least
\[
  \left(\frac{t_-}{t}-\frac{kA}{D}
       \left(\frac{t_-}{t}+2\right)\right)_+
  =\left(1-q-\frac{kA}{D}(3-q)\right)_+.
\]
This proves~\eqref{eq:comparison-lower}.
Both constructions use only inclusions from an earlier point set into a
later one. Their comparison times depend on $E,k,t$, but not on $D$.
For $D$ in a bounded range, one time threshold also makes all the intervals
shorter than one. This proves the asserted uniformity.
\end{proof}
\section{Counts in Short Intervals}
We now iterate Lemma~\ref{lem:comparison}. We first transfer the span bounds
from scale $r$ to an intermediate scale $K=\sqrt{Ar}$. For a comparison
with $E$ between $D$ and $2D$, the two errors are of orders $D/(kr)$ and
$kA/D$. Balancing them suggests $k\approx D/K$, making each error of order
$\sqrt{A/r}$. There are $O(\log(r/A))$ such comparisons. One final
comparison then turns the relative error at scale $K$ into an additive
error of $3A$ plus a smaller term on shorter intervals.
All constants in this section are absolute.
\begin{proposition}\label{prop:short}
There are constants $C_0,C_1>0$ with the following property. Suppose that
\eqref{eq:span-assumption} holds for all sufficiently large $t$, with
$A\geq1$ and $r\geq C_0A$. Set
\[
  \Lambda=\log(r/A),
  \qquad
  S=\frac{\sqrt{Ar}}{\Lambda^2}.
\]
Then, for every sufficiently large $t$,
\begin{equation}\label{eq:short-count}
  \sup_{x\in\T}\sup_{0\leq D\leq S}
  \left|N_t((x,x+D/t])-D\right|
  \leq3A+\frac{C_1A}{\Lambda}.
\end{equation}
The time threshold may depend on $r,A$, and the point sequence, but not
on $x$ or $D$.
\end{proposition}
\begin{proof}
Put
\[
  \theta=\sqrt{A/r},
  \qquad
  K=\sqrt{Ar}.
\]
We take $C_0$ sufficiently large throughout the proof; in particular,
$\theta\leq1/12$ and $K<r-A$.
\paragraph{The terminal bounds.}
An interval of length $(r-A)/t$ contains at most $r$ points. Indeed,
$r+1$ points in such a half-open interval would give an $r$-span of
strictly smaller length, contrary to~\eqref{eq:span-assumption}. Similarly,
an interval of length $(r+A)/t$ contains at least $r$ points: start at the
last point at or before its left endpoint and use the upper span bound.
Therefore
\begin{equation}\label{eq:terminal}
  U_t(r-A)\leq\frac r{r-A},
  \qquad
  V_t(r+A)\geq\frac r{r+A}
\end{equation}
for all sufficiently large $t$.
\paragraph{Transferring the terminal bounds to $K$.}
Construct a list of scales by starting at $K$ and doubling until reaching
$r-A$ for the upper estimate, and $r+A$ for the lower estimate. Shorten
the last step in each list to end at its specified terminal scale.
The estimates propagate from these known terminal bounds back down to
$K$. For two adjacent scales $D,E$, we have
\[
  K\leq D\leq E\leq2D.
\]
Choose $k=\lceil D/K\rceil$ in Lemma~\ref{lem:comparison}. Since
$k\leq2D/K$,
\[
  q=\frac{E}{kr}\leq2\theta,
  \qquad
  \frac{3kA}{D}\leq6\theta.
\]
The upper multiplier is at most
\[
  (1+6\theta)(1+2\theta)\leq1+9\theta,
\]
and the lower multiplier is at least $1-8\theta>0$.
Let $h_U$ and $h_V$ be the numbers of comparisons in the two lists. Both
are at most $C\Lambda$ for an absolute constant $C$. The terminal times
are obtained from $t$ by finite products of positive factors $1+q$ or
$1-q$. Thus, for sufficiently large $t$, every comparison and both
terminal bounds apply. Iteration gives the explicit estimates
\[
  U_t(K)\leq\frac r{r-A}(1+9\theta)^{h_U},
  \qquad
  V_t(K)\geq\frac r{r+A}(1-8\theta)^{h_V}.
\]
Since $A/r=\theta^2$ and $\theta\Lambda\to0$ as $r/A\to\infty$, these
imply
\begin{equation}\label{eq:middle}
  U_t(K)\leq1+O(\theta\Lambda),
  \qquad
  V_t(K)\geq1-O(\theta\Lambda)
\end{equation}
for all sufficiently large $t$. The implied constants are independent
of $r,A$, and the point sequence.
\paragraph{Transferring the bounds at $K$ to shorter intervals.}
Apply Lemma~\ref{lem:comparison} once more, with $E=K$ and $k=1$, so that
$q=\theta$. For $D>0$ and $I=(x,x+D/t]$, it gives
\begin{align*}
  N_t(I)&\leq
    (D+3A)(1+\theta)\,U_{(1+\theta)t}(K),\\
  N_t(I)&\geq
    \bigl(D(1-\theta)-(3-\theta)A\bigr)_+\,
    V_{(1-\theta)t}(K).
\end{align*}
Using~\eqref{eq:middle}, we obtain
\begin{equation}\label{eq:additive-error}
  |N_t(I)-D|
  \leq3A+O\bigl((D+A)\theta\Lambda\bigr).
\end{equation}
For the lower estimate, if the expression inside the positive part is
nonpositive, then $D\leq3A+O(A\theta)$ and the trivial bound $N_t(I)\geq0$
suffices. Otherwise the lower bound in~\eqref{eq:middle} applies directly.
The scale $S=K/\Lambda^2$ is chosen to make the remaining error
$O(A/\Lambda)$. Indeed, for $D\leq S$,
\[
  D\theta\Lambda\leq\frac A\Lambda,
  \qquad
  A\theta\Lambda=O(A/\Lambda),
\]
where the second estimate follows from $\theta\Lambda^2\to0$.
This proves~\eqref{eq:short-count} after fixing sufficiently large absolute
$C_0$ and $C_1$.
The scale lists use only finitely many predetermined comparison times,
and the last comparison uses $(1\pm\theta)t$, independently of $D$.
Consequently a single time threshold works for every $x$ and
$0<D\leq S$. The case $D=0$ is immediate.
\end{proof}
\section{A Finite Sequence in a Short Interval}
For a finite list $z_1,\ldots,z_L\in[0,1)$, define its maximum prefix
counting error by
\[
  H_L(z_1,\ldots,z_L)
  =\max_{1\leq j\leq L}\sup_{0\leq u\leq1}
   \left|\#\{i\leq j:z_i<u\}-ju\right|.
\]
We need a lower bound for every sufficiently long finite list, not merely
a statement about infinitely many prefixes of an infinite sequence.
\begin{theorem}[Finite-prefix discrepancy bound; Larcher]\label{thm:larcher}
There is an absolute integer $L_0$ such that, for every integer $L\geq L_0$
and every list $z_1,\ldots,z_L\in[0,1)$,
\[
  H_L(z_1,\ldots,z_L)\geq\frac1{16}\log L.
\]
\end{theorem}
\begin{proof}[Derivation from Larcher's proof]
Section~3 of~\cite{larcher}, specifically pp.~12--13 of the preprint,
begins with a finite list of length $N=\lfloor a^h\rfloor$, where
$3<a<4$ and $h\in\N$. It proves
\[
  H_N\geq c_a\log N,
  \qquad
  c_a=\frac{(a-2)(8a+3)}{16(1-2a)^2\log a}.
\]
Take $a=7/2$. Then
\[
  c_a=\frac{31}{384\log(7/2)}>0.064>\frac1{16}.
\]
Given an arbitrary list of length $L$, restrict it to the largest such
$N\leq L$. Since $\log N=\log L-O_a(1)$,
\[
  H_L\geq H_N\geq c_a\log L-O_a(1).
\]
The additive constant is independent of the list. The strict margin
$c_a>1/16$ therefore gives the stated bound for all sufficiently large
$L$, with an absolute threshold.
\end{proof}
\begin{lemma}\label{lem:finite-sequence}
Suppose that $B\geq1$, $S\geq2$, and, for all sufficiently large integers
$n$,
\begin{equation}\label{eq:count-hypothesis}
  \left|N_n((x,x+D/n])-D\right|\leq B
  \qquad(x\in\T,\ 0\leq D\leq S).
\end{equation}
If $\lfloor S\rfloor\geq L_0$, then
\[
  B\geq\frac1{16}\log\lfloor S\rfloor.
\]
\end{lemma}
\begin{proof}
Put $L=\lfloor S\rfloor$. Choose $n_0\geq2$ so that
\eqref{eq:count-hypothesis} holds for all $n\geq n_0$, and let $\delta>0$
be the minimum circular distance between distinct points of $P_{n_0}$.
Choose an integer $N>\max\{L,n_0\}$ so large that $L/N<\delta$.
The average length of a cyclic $L$-span of $P_N$ is $L/N$. Hence there is
such a span of length $\ell\leq L/N$. Its half-open interval contains
exactly $L$ points. Translate it slightly forward, keeping those points
inside and placing both endpoints outside $P_N$. Denote the resulting
interval by $J=(a,a+\ell]$. Since $\ell<\delta$, the interval $J$ contains
at most one point of $P_{n_0}$.
List the $L$ points in $J$ in their order of insertion and rescale $J$ to
$(0,1)$, obtaining $z_1,\ldots,z_L$. Consider the first $j$ points, and
let $n\leq N$ be the insertion time of the $j$th one. If $n<n_0$, then
$j\leq1$, so this prefix has counting error at most $1\leq B$.
Now suppose $n\geq n_0$ and define
\[
  f(u)=N_n((a,a+u\ell])-n\ell u
  \qquad(0\leq u\leq1).
\]
Applying~\eqref{eq:count-hypothesis} to the two complementary subintervals
of $J$ gives
\[
  |f(u)|\leq B,
  \qquad
  |f(1)-f(u)|\leq B,
\]
because their lengths multiplied by $n$ are at most
$n\ell\leq N\ell\leq L\leq S$. Since $f(1)=j-n\ell$,
\[
  \begin{aligned}
    \#\{i\leq j:z_i\leq u\}-ju
    &=f(u)-uf(1)\\
    &=(1-u)f(u)-u\bigl(f(1)-f(u)\bigr).
  \end{aligned}
\]
Therefore
\[
  \left|\#\{i\leq j:z_i\leq u\}-ju\right|
  \leq(1-u)B+uB=B.
\]
This is why control of both complementary intervals preserves $B$,
rather than giving $2B$.
Every prefix of the rescaled list thus has counting error at most $B$.
Since the points lie in $(0,1)$, the conventions $z_i<u$ and $z_i\leq u$
give the same supremum by one-sided limits. Hence
$H_L(z_1,\ldots,z_L)\leq B$, and Theorem~\ref{thm:larcher} gives the
conclusion.
\end{proof}
\section{Proof of the Ratio Bound}\label{sec:ratio-proof}
\begin{proof}[Proof of the ratio assertion in Theorem~\ref{thm:main}]
Fix a sufficiently large $r$, and suppose for a contradiction that
\[
  \limsup_{n\to\infty}\frac{M_n^{(r)}}{m_n^{(r)}}
  <1+\frac{\log r}{100r}.
\]
Put $A=(\log r)/100$, so $A\geq1$ for sufficiently large $r$.
Choose $0<C<A$ such that
\[
  \frac{M_n^{(r)}}{m_n^{(r)}}\leq1+\frac Cr
\]
for all sufficiently large $n$. Suppress the superscript $(r)$ for the
rest of the proof. The average-span identity gives $nm_n\leq r\leq nM_n$,
and therefore
\[
  n(M_n-m_n)\leq C,
  \qquad
  nM_n\leq r+C.
\]
For $n=\lfloor t\rfloor$, define
\[
  a_t=r-nm_n,
  \qquad
  b_t=tM_n-r.
\]
Both are nonnegative, and
\[
  \frac{r-a_t}{t}=\frac nt\,m_n\leq m_n
  \leq M_n=\frac{r+b_t}{t}.
\]
Moreover,
\[
  a_t+b_t
  =n(M_n-m_n)+(t-n)M_n
  \leq C+\frac{r+C}{n}<A
\]
for all sufficiently large $t$. Thus~\eqref{eq:span-assumption} holds.
Since $r/A\to\infty$, Proposition~\ref{prop:short} applies for sufficiently
large $r$. It gives~\eqref{eq:count-hypothesis} with
\[
  S=\frac{\sqrt{Ar}}{\log^2(r/A)},
  \qquad
  B=3A+\frac{C_1A}{\log(r/A)}.
\]
With $A=(\log r)/100$, we have $S\to\infty$ and
\[
  B=\frac3{100}\log r+O(1),
  \qquad
  \log\lfloor S\rfloor
  =\frac12\log r-\frac32\log\log r+O(1),
\]
with absolute implied constants. Lemma~\ref{lem:finite-sequence} now gives
\[
  \frac3{100}\log r+O(1)
  \geq\frac1{16}\log\lfloor S\rfloor
  =\frac1{32}\log r-O(\log\log r).
\]
This is impossible for sufficiently large $r$, since $3/100<1/32$.
All constants governing the required size of $r$ are absolute, so the
threshold for $r$ is independent of the point sequence. This proves the
ratio assertion.
\end{proof}
\begin{remark}
The coefficient $1/100$ is chosen for simplicity; we have not optimized
the constant.
\end{remark}
\section{One-Sided Bounds in Spatial \texorpdfstring{$L^1$}{L1}}
The ratio argument used pointwise control of every $r$-span. A one-sided
bound does not give pointwise control in the opposite direction, but the
average-span identity supplies a useful substitute: it controls the total
deviation of all the spans. Averaging the walk comparison in space then
preserves this $L^1$ control.

Fix one of the following two eventual assumptions:
\begin{equation}\label{eq:one-sided-assumption}
  nM_n^{(r)}-r\leq A
  \qquad\text{or}\qquad
  r-nm_n^{(r)}\leq A,
\end{equation}
where $A\geq1$ is fixed and one fixed alternative holds for every
sufficiently large integer $n$. The same argument applies in both cases.

For $p\in P_t$, let $L_{t,k}(p)$ be the clockwise distance from $p$ to the
point $kr$ places after it in the cyclic order of $P_t$. Here $r$ and $k$
are fixed, and $t$ is sufficiently large that $kr<|P_t|$.
\begin{lemma}\label{lem:one-sided-span-l1}
Under either assumption in~\eqref{eq:one-sided-assumption},
\begin{equation}\label{eq:walk-l1}
  \sum_{p\in P_t}
  \left|L_{t,k}(p)-\frac{kr}{t}\right|
  \leq 2kA+\frac{kr}{t}.
\end{equation}
\end{lemma}
\begin{proof}
First take an integer time $n$. If $S_1,\ldots,S_n$ are the $r$-spans,
then
\[
  \sum_{i=1}^n\left(S_i-\frac rn\right)=0.
\]
Under the first assumption, every summand is at most $A/n$, so the sum of
the positive summands is at most $A$. Under the second assumption, every
summand is at least $-A/n$, so the sum of the negative parts is at most
$A$. In either case, the zero-sum identity gives
\begin{equation}\label{eq:r-span-total-error}
  \sum_{i=1}^n\left|S_i-\frac rn\right|\leq2A.
\end{equation}
Now let $n=\lfloor t\rfloor$. Replacing $r/n$ by $r/t$ on the left of
\eqref{eq:r-span-total-error} costs at most
\[
  n\left|\frac rn-\frac rt\right|
  =\frac{r(t-n)}t=o_{t\to\infty}(1).
\]
A $kr$-span is a sum of $k$ suitably shifted $r$-spans. The triangle
inequality, followed by summation over the cyclic starting point, proves
\eqref{eq:walk-l1}. The same estimate is uniform when $t$ ranges over any
fixed multiplicative interval.
\end{proof}

For $D\geq0$, set
\[
  \Delta_t(x,D)=N_t((x,x+D/t])-D
\]
and define its positive spatial mass by
\[
  Z_t(D)=\int_{\T}\bigl(\Delta_t(x,D)\bigr)_+\,dx.
\]
At integer times the mean of $\Delta_t(\,\cdot\,,D)$ is zero, and hence
\begin{equation}\label{eq:q-equals-2z}
  \int_{\T}|\Delta_n(x,D)|\,dx=2Z_n(D).
\end{equation}
The next lemma is the averaged counterpart of
Lemma~\ref{lem:comparison}.

\begin{lemma}\label{lem:l1-comparison}
Assume~\eqref{eq:one-sided-assumption}. Fix $D,E>0$ and an integer
$k\geq1$, and put
\[
  q=\frac{E}{kr}.
\]
If $q<1$, then
\begin{equation}\label{eq:l1-comparison}
  Z_t(D)
  \leq \frac{(1+q)D}{E}\,Z_{(1+q)t}(E)
       +qD+8kA+\frac{4kr}{t}.
\end{equation}
\end{lemma}
\begin{proof}
Put $t_+=(1+q)t$ and $\ell=E/t_+$. For $0\leq u\leq\ell$, choose
$s=s(u)$ so that
\[
  \frac{kr}{t}-\frac{kr}{s}=u.
\]
The identity $q=E/(kr)$ ensures that $t\leq s\leq t_+$. As in the proof
of Lemma~\ref{lem:comparison}, use the injection
\[
  T_u:\quad
  P_t\xrightarrow{F_t}P_t
  \hookrightarrow P_s\xrightarrow{B_s}P_s
  \hookrightarrow P_{t_+},
\]
where $F_t$ and $B_s$ move forward and backward by $kr$ places.
Choose compatible lifts to $\R$ and write
\[
  T_u(p)=p+u+\eta_u(p).
\]
Lemma~\ref{lem:one-sided-span-l1}, first at $t$ and then at $s$, gives
\begin{equation}\label{eq:transport-cost}
  \sum_{p\in P_t}|\eta_u(p)|
  \leq4kA+\frac{2kr}{t},
\end{equation}
uniformly for $0\leq u\leq\ell$. In the second use of the lemma, the
relevant predecessors $B_s(F_t(p))$ form a subset of $P_s$, so their
total contribution is no larger than the full sum.
Moving a single atom by circular distance $|\eta|$ changes the interval
membership function, in spatial $L^1$, by at most $2|\eta|$. Thus, if
$I_x=(x,x+D/t]$, injectivity of $T_u$ and
\eqref{eq:transport-cost} give
\[
  N_t(I_x)\leq N_{t_+}(I_x+u)+R_u(x),
  \qquad
  \int_{\T}R_u(x)\,dx\leq8kA+\frac{4kr}{t},
\]
where $R_u\geq0$. Average over $0\leq u\leq\ell$. Fubini's theorem gives
\[
  \frac1\ell\int_0^\ell N_{t_+}(I_x+u)\,du
  =\frac1\ell\int_{I_x}N_{t_+}((v,v+\ell])\,dv.
\]
Since $|I_x|E/\ell=(1+q)D$, it follows that
\[
  \Delta_t(x,D)
  \leq qD+\frac1\ell\int_{I_x}\Delta_{t_+}(v,E)\,dv+R(x),
\]
where $R\geq0$ and
$\int_{\T}R\leq8kA+4kr/t$. Taking positive parts and integrating in $x$
yields
\[
  Z_t(D)
  \leq qD+\frac{|I_x|}{\ell}Z_{t_+}(E)
       +8kA+\frac{4kr}{t}.
\]
Because $|I_x|/\ell=(1+q)D/E$, this is
\eqref{eq:l1-comparison}.
\end{proof}

There is also a simple terminal estimate at the original scale $r$.
\begin{lemma}\label{lem:l1-terminal}
Under~\eqref{eq:one-sided-assumption},
\[
  Z_t(r)\leq A
\]
for every sufficiently large $t$.
\end{lemma}
\begin{proof}
Write the points of $P_t$ in cyclic order as $y_1,\ldots,y_n$, where
$n=\lfloor t\rfloor$. Apart from endpoints,
\[
  N_t((x,x+r/t])
  =\sum_{i=1}^n\mathbf 1_{(y_i-r/t,y_i]}(x).
\]
On the other hand, the $r$-span arcs cover almost every point of the
circle exactly $r$ times, so
\[
  r=\sum_{i=1}^n\mathbf 1_{(y_{i-r},y_i]}(x).
\]
Pairing the arcs with the same right endpoint shows that
\[
  \int_{\T}|\Delta_t(x,r)|\,dx
  \leq\sum_{i=1}^n
  \left|S_{i-r}(t)-\frac rt\right|
  \leq2A+\frac{r(t-n)}t.
\]
Moreover,
\[
  \int_{\T}\Delta_t(x,r)\,dx
  =\frac{rn}{t}-r=-\frac{r(t-n)}t.
\]
The positive mass is half the sum of the $L^1$ norm and the mean.
Consequently $Z_t(r)\leq A$.
\end{proof}

We may now transfer the terminal estimate to much shorter intervals.
\begin{proposition}\label{prop:short-l1}
There are absolute constants $C_2,C_3>0$ with the following property.
Suppose that~\eqref{eq:one-sided-assumption} holds with $A\geq1$ and
$r\geq C_2A$. Put
\[
  \Lambda=\log(r/A),
  \qquad
  S=\frac{\sqrt{Ar}}{\Lambda^2}.
\]
Then, for all sufficiently large integers $n$,
\begin{equation}\label{eq:short-count-l1}
  \sup_{0\leq D\leq S}
  \int_{\T}\left|N_n((x,x+D/n])-D\right|\,dx
  \leq C_3A.
\end{equation}
The time threshold may depend on $r,A$, and the sequence, but not on $D$.
\end{proposition}
\begin{proof}
Write
\[
  \theta=\sqrt{A/r},
  \qquad
  K=\sqrt{Ar},
  \qquad
  z_t(D)=\frac{Z_t(D)}D \quad(D>0).
\]
Starting at $K$, double the scale until reaching $r$, shortening the last
step if necessary. Thus consecutive scales $D,E$ satisfy
\[
  K\leq D\leq E\leq2D,
\]
and the number of steps is $O(\Lambda)$. For such a pair choose
$k=\lceil D/K\rceil$ in Lemma~\ref{lem:l1-comparison}. Then
\[
  q=\frac{E}{kr}\leq2\theta,
  \qquad
  \frac{kA}{D}\leq2\theta.
\]
After division by $D$, the comparison lemma gives, on this finite scale
chain,
\begin{equation}\label{eq:z-recursion}
  z_t(D)\leq(1+C\theta)z_{(1+q)t}(E)+C\theta
            +\frac{4kr}{tD}
\end{equation}
with an absolute constant $C$. Lemma~\ref{lem:l1-terminal} gives
$z_t(r)\leq A/r=\theta^2$. Iterating~\eqref{eq:z-recursion} through the
$O(\Lambda)$ scales, and taking $C_2$ large enough that
$\theta\Lambda$ is small, yields
\begin{equation}\label{eq:z-middle}
  z_t(K)\leq C'\theta\Lambda
\end{equation}
for every sufficiently large $t$.
Finally apply Lemma~\ref{lem:l1-comparison} with $E=K$ and $k=1$. Here
$q=\theta$, so for $0<D\leq S$,
\[
  \begin{aligned}
  Z_t(D)
  &\leq \frac{(1+\theta)D}{K}Z_{(1+\theta)t}(K)
       +\theta D+8A+\frac{4r}{t}\\
  &\leq 8A+C'D\theta\Lambda+\frac{4r}{t}.
  \end{aligned}
\]
The definition of $S$ gives $D\theta\Lambda\leq A/\Lambda$. At an
integer time $n$, identity~\eqref{eq:q-equals-2z} now proves
\eqref{eq:short-count-l1}, after increasing the absolute constant $C_3$.
Indeed, we may take $n$ late enough that $4r/n\leq A$. The transport error
above is independent of $D$, so one late time works
uniformly for $0<D\leq S$; the case $D=0$ is immediate.
\end{proof}
\section{Localizing the Spatial Error}
To turn Proposition~\ref{prop:short-l1} into a contradiction, we use the
$L^1$ form of the two-dimensional discrepancy theorem. For a finite set
$\mathcal P\subset[0,1]^2$, write
\[
  D_{\mathcal P}(u,v)
  =\#\bigl(\mathcal P\cap((0,u]\times(0,v])\bigr)
   -|\mathcal P|uv.
\]
Endpoint conventions are immaterial in the following integral.

\begin{theorem}[Hal\'asz]\label{thm:halasz}
There is an absolute constant $c_H>0$ such that every set $\mathcal P$ of
$M\geq2$ points in $[0,1]^2$ satisfies
\[
  \int_0^1\int_0^1|D_{\mathcal P}(u,v)|\,du\,dv
  \geq c_H\sqrt{\log M}.
\]
\end{theorem}
This is the unnormalized form of Hal\'asz's planar $L^1$ discrepancy
theorem~\cite{halasz}. We need the following direct consequence.

\begin{lemma}\label{lem:l1-localization}
There are absolute constants $c_4>0$ and $S_0$ with the following
property. Suppose $S\geq S_0$, $B\geq1$, and, for all sufficiently large
integers $n$,
\begin{equation}\label{eq:l1-count-hypothesis}
  \int_{\T}\left|N_n((x,x+D/n])-D\right|\,dx\leq B
  \qquad(0\leq D\leq S).
\end{equation}
Then
\[
  B\geq c_4\sqrt{\log S}.
\]
\end{lemma}
\begin{proof}
Put $L=\lfloor S\rfloor$, choose a large integer $N>L$, and set $w=L/N$.
For each $a\in\T$, let
\[
  M_a=N_N((a,a+w]).
\]
From~\eqref{eq:l1-count-hypothesis} with $n=N$ and $D=L$,
\begin{equation}\label{eq:ma-average}
  \int_{\T}|M_a-L|\,da\leq B.
\end{equation}
View the oriented arc $(a,a+w]$ as a copy of $(0,1]$. From every
$x_i\in P_N\cap(a,a+w]$, form the point
\[
  \left(\frac{x_i-a}{w},\frac{i}{N}\right)\in[0,1]^2,
\]
where the first coordinate is measured along the oriented arc. Denote
the resulting $M_a$-point set by $\mathcal P_a$, and put
\[
  G_a(u,v)
  =N_{\lfloor Nv\rfloor}((a,a+uw])-Luv.
\]
Apart from a null set of $(u,v)$ caused by endpoints,
\[
  D_{\mathcal P_a}(u,v)
  =G_a(u,v)+(L-M_a)uv.
\]
Theorem~\ref{thm:halasz} and
$\int_0^1\int_0^1uv\,du\,dv=1/4$ therefore give
\begin{equation}\label{eq:ga-lower}
  \|G_a\|_{L^1([0,1]^2)}
  \geq c_H\sqrt{\log M_a}-\frac14|M_a-L|
\end{equation}
whenever $M_a\geq2$.
If $B\geq L/4$, the conclusion is immediate for large $L$. Otherwise,
Markov's inequality and~\eqref{eq:ma-average} show that
\[
  \bigl|\{a:L/2\leq M_a\leq3L/2\}\bigr|\geq\frac12.
\]
Integrating~\eqref{eq:ga-lower} over this set gives
\begin{equation}\label{eq:integrated-lower}
  \int_{\T}\|G_a\|_1\,da
  \geq\frac{c_H}{2}\sqrt{\log(L/2)}-\frac B4.
\end{equation}
It remains to bound the same integral from above. Let $n_0$ be a threshold
for~\eqref{eq:l1-count-hypothesis}, and put
$n=\lfloor Nv\rfloor$. For $n\geq n_0$, set
$D=nuw\leq L\leq S$. Then
\[
  G_a(u,v)
  =\bigl(N_n((a,a+uw])-nuw\bigr)
    +Lu\left(\frac nN-v\right).
\]
The second term has absolute value at most $w$, while the integral in $a$
of the absolute value of the first is at most $B$ by
\eqref{eq:l1-count-hypothesis}. On the remaining strip
$0\leq v<n_0/N$, we have $|G_a(u,v)|\leq n_0+L$. Consequently,
\begin{equation}\label{eq:integrated-upper}
  \int_{\T}\|G_a\|_1\,da
  \leq B+\frac LN+\frac{n_0(n_0+L)}N
  =B+o_{N\to\infty}(1).
\end{equation}
Combining~\eqref{eq:integrated-lower} and
\eqref{eq:integrated-upper}, then letting $N\to\infty$, gives
\[
  \frac54B\geq\frac{c_H}{2}\sqrt{\log(L/2)}.
\]
Since $L=\lfloor S\rfloor$, this proves the lemma after adjusting the
absolute constants.
\end{proof}
\section{Proof of the One-Sided Bounds}
\begin{proof}[Proof of the one-sided assertions in
Theorem~\ref{thm:main}]
Let $C_3$ and $c_4$ be the constants in
Proposition~\ref{prop:short-l1} and Lemma~\ref{lem:l1-localization}, and
choose an absolute $c>0$ sufficiently small.
Suppose first, for a contradiction, that
\[
  \limsup_{n\to\infty}\bigl(nM_n^{(r)}-r\bigr)
  <c\sqrt{\log r}.
\]

Put
\[
  A=c\sqrt{\log r},
  \qquad
  S=\frac{\sqrt{Ar}}{\log^2(r/A)}.
\]
For sufficiently large $r$, we have $A\geq1$,
$r\geq C_2A$, and the first alternative in
\eqref{eq:one-sided-assumption} holds for every sufficiently large $n$.
For sufficiently large $r$, we also have $S\geq S_0$ and $C_3A\geq1$.
Proposition~\ref{prop:short-l1}, followed by
Lemma~\ref{lem:l1-localization}, gives
\begin{equation}\label{eq:final-one-sided-comparison}
  C_3A\geq c_4\sqrt{\log S}.
\end{equation}
For this choice of $A$,
\[
  \log S=\frac12\log r+O(\log\log r),
\]
and in particular $\log S\geq(\log r)/3$ for all sufficiently large $r$.
Inequality~\eqref{eq:final-one-sided-comparison} is impossible if, for
example, $c<c_4/(2C_3\sqrt3)$. This proves the first one-sided assertion.
If instead
\[
  \limsup_{n\to\infty}\bigl(r-nm_n^{(r)}\bigr)
  <c\sqrt{\log r},
\]

the identical argument uses the second alternative in
\eqref{eq:one-sided-assumption}. This proves the second assertion, with
the same absolute constants.
\end{proof}
\section*{Acknowledgments}
The author supplied the two main ideas behind the ratio argument:
using injective compositions of forward and backward cyclic walks to
bound point counts in short intervals, and applying a quantitative form
of Schmidt's sequence-discrepancy theorem to obtain a logarithmic lower
bound for the gap ratio. The second idea was motivated by the order of
the upper bound of Cl\'ement and Steinerberger.

The one-sided argument uses the spatial $L^1$ form of the same cyclic-walk
comparison, followed by a localization to planar discrepancy.

GPT Astra was used for the literature search that
identified Larcher's quantitative discrepancy bound and for completing
the mathematical argument from these ideas. It was also used to develop
and audit the $L^1$ transport and localization argument, to review
the proof, and to revise the exposition. The author independently checked the
arguments and calculations and
assumes full responsibility for all mathematical claims.

\end{document}